\documentclass{amsart}%
\usepackage{amssymb}
\usepackage{amsfonts}
\usepackage{amsmath}
\usepackage{graphicx}\usepackage{xy}

\xyoption{all}

\providecommand{\U}[1]{\protect\rule{.1in}{.1in}}
\newtheorem{theorem}{Theorem}

\newtheorem{corollary}[theorem]{Corollary}

\newtheorem{proposition}[theorem]{Proposition}

\theoremstyle{remark}
\newtheorem{remark}[theorem]{Remark}

\title{Comments on the height reducing property II}

\author{Shigeki AKIYAMA}
\address{Institute of Mathematics, University of Tsukuba, 1-1-1 Tennodai, Tsukuba,
Ibaraki, 350-0006 Japan}
\email{akiyama@math.tsukuba.ac.jp}
\author{J\"{o}rg M. THUSWALDNER}
\address{Department of mathematics and statistics, Leoben University,
Franz-Josef-Strasse 18, A-8700, Leoben, Austria}
\email{joerg.thuswaldner@unileoben.ac.at}
\author{Toufik ZA\"IMI}
\address{Department of mathematics and informatics, Larbi Ben M'hidi University, Oum El
Bouaghi 04000,  Algeria}
\email{toufikzaimi@yahoo.com}



\thanks{The second author was supported by projects I1136 and W1230 funded by
the Austrian Science Fund.}

\date{\today}

\begin{document}

\begin{abstract}
A complex number $\alpha$ is said to satisfy the height reducing property if there is a finite set $F\subset \mathbb{Z}$ such that $\mathbb{Z}[\alpha]=F[\alpha]$, where $\mathbb{Z}$ is the ring of the rational integers. It is easy to see that $\alpha$ is an algebraic number when it satisfies the height reducing property. We prove the relation $\operatorname{Card}(F)\geq \max\{2,\left\vert M_{\alpha}(0)\right\vert \},$  where $M_{\alpha}$ is the minimal polynomial of $\alpha$ over the field of the rational numbers, and discuss the related optimal cases, for some classes of algebraic numbers $\alpha$. In addition, we show that there is an algorithm to determine the minimal height polynomial of a given algebraic number, provided it has no conjugate of modulus one.
\end{abstract}

\maketitle

\section{Introduction}

We continue, in this manuscript, the study of the numbers with height reducing property, in short HRP. Recall that a complex number $\alpha$ is said to satisfy HRP if there is a finite set $F\subset \mathbb{Z},$ such that each polynomial with coefficients in $\mathbb{Z}$, evaluated at $\alpha$, belongs to the set
\[
F[\alpha]:=\left\{%
%TCIMACRO{\dsum \limits_{j=0}^{n}}%
%BeginExpansion
{\displaystyle\sum\limits_{j=0}^{n}}
%EndExpansion
f_{j}\alpha^{j}\mid(f_{0},\ldots,f_{n})\in F^{n+1},\text{ }n\in\mathbb{N}\right\},
\]
(see {\em e.g.}~\cite{ADJ:12,ATZ:14,AZ:13}). In other
words, $\alpha$ satisfies HRP when $\mathbb{Z}[\alpha]$ may be reduced to $F[\alpha]$. In \cite{ATZ:14} it is proved that $\alpha$ satisfies HRP if and only if $\alpha$ is an algebraic number whose conjugates (including $\alpha$ itself) are either all of modulus one, or all of modulus greater than one.

Throughout this paper, when we speak about conjugates, the minimal polynomial and the degree of
an algebraic number $\alpha$, this is meant over the field of the rational
numbers $\mathbb{Q}.$ The minimal polynomial $M_{\alpha}$ of $\alpha$ is
supposed to be primitive, that is, the coefficients of $M_{\alpha}$ are
integers whose greatest common divisor is one.

In fact the equality $\mathbb{Z}[\alpha]=S[\alpha],$ where $S$ is a subset of
the complex field $\mathbb{C},$ implies trivially the relations $S\subset
\mathbb{Z}[\alpha]$ and $\mathbb{Z}[\alpha]=(-S)[\alpha].$ The following
result shows that a complex number $\alpha$ satisfies HRP if and only if 
\begin{equation}\label{eq:(1)}
\mathbb{Z}[\alpha]=S[\alpha]\text{\ with }S \subset \mathbb{C}\text{\ finite.}
\end{equation}

\begin{theorem}\label{th:1}
If \eqref{eq:(1)} holds for some pair $(\alpha,S)$, then there is a finite subset $F \subset \mathbb{Z}$,  such that $\mathbb{Z}[\alpha]=F[\alpha]$ and
\[
\operatorname{Card}(F)\leq \operatorname{Card}(S)(\operatorname{Card}(S)^{s+1}-1)/(\operatorname{Card}(S)-1),
\]
where $s$ is the greatest exponent of $\alpha$ of some fixed choice of representations of the elements of $S$ in $\mathbb{Z}[\alpha]$. Moreover, $\operatorname{Card}(S)\geq2$ and the set $S$ contains at least $\left\vert M_{\alpha}(0)\right\vert$ elements of the form $P_{j}(\alpha),$ where $j\in\{0,\ldots,\left\vert M_{\alpha}(0)\right\vert -1\},$ $P_{j}\in
\mathbb{Z}[x]$ and $P_{j}(0)\equiv j\operatorname{mod}M_{\alpha}(0).$
\end{theorem}

For a given number $\alpha$ satisfying \eqref{eq:(1)}, we denote by $S_{\alpha}$ a fixed
choice for $S\subset\mathbb{C},$ having the minimal number of elements. We also designate by $\mathcal{F}_{N}$ the set of those algebraic numbers $\alpha$ which satisfy $\operatorname{Card}(S_{\alpha})=N$. Note that, using this notation, $\alpha\in\mathcal{F}_{N}$ for some $N$ is equivalent to $\alpha$ satisfying HRP. It follows immediately
from the second assertion in Theorem~\ref{th:1} that $S_{\alpha}$ contains a complete residue system, in short CRS, $\operatorname{mod}\alpha$ in $\mathbb{Z}[\alpha]$, thus
\begin{equation}\label{eq:(2)}
\max\{2,\left\vert M_{\alpha}(0)\right\vert \}\leq \operatorname{Card}(S_{\alpha})
\end{equation}
and the index $N$ in the notation $%
%TCIMACRO{\tciFourier}%
%BeginExpansion
\mathcal{F}%
%EndExpansion
_{N}$ is at least $2$. A result of Lagarias and Wang~\cite{LW:97} implies that an expanding algebraic integer $\alpha,$ that is an algebraic integer whose conjugates are of modulus greater than one, satisfies \eqref{eq:(1)} with $S=\{0,\pm
1,\ldots,\pm(\left\vert M_{\alpha}(0)\right\vert -1)\}.$ It thus follows for expanding integers $\alpha$ that
\begin{equation}\label{eq:(3)}
\left\vert M_{\alpha}(0)\right\vert \leq \operatorname{Card}(S_{\alpha})\leq2\left\vert
M_{\alpha}(0)\right\vert -1,
\end{equation}
and so $\alpha\in%
%TCIMACRO{\tciFourier}%
%BeginExpansion
\mathcal{F}%
%EndExpansion
_{N}$ for some $N$ $\in\{\left\vert M_{\alpha}(0)\right\vert ,\ldots,2\left\vert M_{\alpha}(0)\right\vert -1\}.$

It is interesting to determine the elements of the optimal set $\mathcal{F}_{2},$ and to characterize all algebraic numbers $\alpha$ satisfying  $\alpha\in\mathcal{F}_{\left\vert M_{\alpha}(0)\right\vert }$. Theorem~\ref{th:2} below collects some partial answers to these questions.

It is easy to see that $\{0,1,\ldots,\left\vert \alpha\right\vert
-1)\}[\alpha]=$\textit{\ }$\mathbb{Z},$ for $\alpha
\in\mathbb{Z\cap(-\infty},-2]$ and $\{-1,0,\ldots,\alpha -2\}[\alpha]=$\textit{\ }$\mathbb{Z}$, for $\alpha\in\mathbb{Z\cap\lbrack
}3,\mathbb{\infty})$, and so, in both cases, $\alpha\in%
%TCIMACRO{\tciFourier}%
%BeginExpansion
\mathcal{F}%
%EndExpansion
_{\left\vert M_{\alpha}(0)\right\vert }.$ This fact is already proved by Gr\"{u}nwald~\cite{G:85}. The case where
$\alpha$ is an expanding integer and $S_{\alpha}=\{0,1,\ldots,\left\vert
M_{\alpha}(0)\right\vert -1\},$ has been considered more than thirty years
ago; such a pair $(\alpha,S_{\alpha})$ has been called a {\em canonical number
system} (of the ring $\mathbb{Z}[\alpha]).$ Many results about canonical number
systems are known (see for instance the references and the results in~\cite{P:04}).
For example, K\'{a}tai and Kov\'{a}cs~\cite{KK:81} showed that a quadratic expanding
integer $\alpha$ gives rise to a canonical number system if and only if
$M_{\alpha}(x)=x^{2}+a_{1}x+a_{2},$ $a_{2}\geq2$ and $-1\leq a_{1}\leq a_{2}.$
This was proved independently by Gilbert~\cite{G:81}.  Kov\'{a}cs~\cite{K:81} showed
that the conditions $M_{\alpha}(x)=x^{d}+a_{1}x^{d-1}+\cdot\cdot
\cdot+a_{d-1}x+a_{d},$ $a_{d}\geq2,$ $d\geq2$ and $1\leq a_{1}\leq\cdot
\cdot\cdot\leq a_{d-1}\leq a_{d}$ are sufficient to obtain a canonical number
system, too. The problem to characterize canonical number systems of a given
degree has later been embedded into the problem of determining shift radix systems, see
{\em e.g.} \cite{ABBPT:05,ABPT:06} for details.

Following the Hungarian tradition, we say that the pair $(\alpha,S_{\alpha})$
is a {\em number system} of the ring $\mathbb{Z}[\alpha],$ when 
$\operatorname{Card}(S_{\alpha})=\max \{2,|M_{\alpha}(0)|\}$ and $0\in S_{\alpha}.$ It is worth noting that number systems have been
defined in a more general context, and many related results are known. For
instance, see \cite{GK:07,vdW:08} for more recent developments. An important general result is due
to K\'{a}tai~\cite{K:99}, who showed that for any number field $K$ there is an
effectively computable constant $c(K)\geq2,$ such that if the conjugates of an
element $\alpha$  of the ring of integers $\mathbb{Z}_{K}$ of $K$, are of
modulus greater than $c(K),$ then $\alpha$ gives rise to a number system of
$\mathbb{Z}_{K}.$ The particular case where $K$ is a real (resp., imaginary)
quadratic field has been specified in~\cite{F:99} (resp., in~\cite{K:94,S:89}). In fact by
considering the companion matrix of the polynomial $M_{\alpha}$ it is easy to
see by a theorem of Germ\'{a}n and Kov\'{a}cs~\cite{GK:07} that we may choose
$c(K)=2$ for any $K,$ without affecting the conclusion; this gives a complete
answer to the above mentioned question when the conjugates of the algebraic
integer $\alpha$ are all of modulus greater than 2.

\medskip\medskip

\begin{theorem}\label{th:2}
Let $\alpha$ be an algebraic number with (primitive) minimal polynomial $M_\alpha\in\mathbb{Z}[x]$. Then the following assertions hold.
\begin{itemize}
\item[(i)] The roots of unity belong to the set $\mathcal{F}_{2},$ and if $\alpha\in\mathcal{F}_{2}$ then $\alpha$ is an algebraic number whose conjugates are all of modulus $1$ or is an expanding integer, with $\left\vert
M_{\alpha}(0)\right\vert =2.$

\item[(ii)] If an algebraic number $\alpha$ satisfies $|M_{\alpha}(1)|=1,$ then $\alpha\notin\mathcal{F}_{\left\vert M_{\alpha}(0)\right\vert}$.

\item[(iii)] Let $\alpha$ be an algebraic number whose conjugates are all of modulus $1$. If $\operatorname{Card}(S_{\alpha})=\left\vert M_{\alpha}(0)\right\vert ,$  then $S_{\alpha}\nsubseteq\mathbb{Z}.$

\item[(iv)]  If $\alpha$ is an algebraic integer whose conjugates are all of modulus greater than $2,$ then $\alpha\in\mathcal{F}_{\left\vert M_{\alpha}(0)\right\vert }$.

\item[(v)]  Let $\alpha=a/b,$ where $a\in\mathbb{N},$ $b\in\mathbb{Z},$ $a>\left\vert b\right\vert \geq1$ and $\gcd(a,b)=1.$ If $a\neq b+1$ (resp., $a=b+1),$ then $\alpha\in\mathcal{F}_{\left\vert M_{\alpha}(0)\right\vert }$ (\textit{resp.,} $\alpha\in\mathcal{F}_{2\left\vert M_{\alpha}(0)\right\vert -1}).$ Moreover, we can choose
$S_{\alpha}$ in a way that $0\in S_{\alpha}\subset\mathbb{Z}.$
\end{itemize}
\end{theorem}

It follows, in particular, by Theorem~\ref{th:2}~(iv), that for each algebraic integer
$\alpha,$\ there is a non-negative rational integer $p$\ such that $\alpha\pm
k\in\mathcal{F}_{\left\vert M_{\alpha\pm k}(0)\right\vert },$ $\forall$ $k\in\mathbb{N}
\cap\lbrack p,\infty).$ This result may also be deduced from \cite[Theorem 5]{P:04}, which uses the above mentioned result of Kov\'{a}cs~\cite{K:81}. Notice also, by Theorem~\ref{th:2}~(v), that there is a number system for any rational number $a/b,$
where $a\in\mathbb{N},$ $b\in\mathbb{Z},$ $a>\left\vert b\right\vert \geq1$, $\gcd(a,b)=1$ and $a\neq b+1.$ See also~\cite{AFS:08} for an investigation of number
systems in rational bases.

The height reducing problem is related to the multiplicity of representations
of an element $z\in\mathbb{Z}[\alpha]$, {\em i.e.}, the number of equivalent
representations of the same number as a polynomial in base $\alpha:$
\[
z=\sum_{i=0}^{\ell}a_{i}\alpha^{i}=\sum_{i=0}^{\ell}b_{i}\alpha^{i}
\]
with $a_{i},b_{i}\in\mathbb{Z}.$ To find all these representations, we clearly need to study representations of 0 of the 
form
\[
\sum_{i=0}^{\ell}(a_{i}-b_{i})\alpha^{i}=0.
\]
If there is $H>0$ such that $|a_{i}-b_{i}|<H$ for all $i\in\mathbb{N}$, then there is a finite automaton which recognizes representations of zero, under some assumption (for basics on automata theory we refer to \cite{E:74}).

\begin{theorem}\label{th:3}
Let $H>0$ and let $\alpha$ be an algebraic number without conjugates on the
unit circle. Then, there is a finite automaton $Z(H)$ which recognizes
words $d_{m}\dots d_{0}\in\{-H,\dots,H\}^{\ast}$ such that 
$\sum_{i=0}^{m}d_{i}\alpha^{i}=0.$
\end{theorem}

This automaton tells the growth rate of the number of equivalent representations,
and is used in the study of the boundary and the topology of fractal tilings,
when $\alpha$ is an expanding algebraic integer~({\em cf.\ e.g.} \cite{BS:05,FT:06,ST:09}). We expect that it also has
potential applications in the study of spectra of polynomials and related
topics (see for instance \cite{AK:13,Z:07,Z:10}). In connection with the height reducing
property, Theorem~\ref{th:3} enables one to determine \textit{\textquotedblleft the minimal
height polynomial}\textquotedblright\ of a given algebraic number $\alpha$,
that is a non-zero polynomial $P\in\mathbb{Z}[x],$ satisfying $P(\alpha)=0$
where the maximum $H(P)$ of the absolute values of the coefficients of $P$ is as small as possible. Increasing one by one the value $H$ until the automaton recognizes a non-empty word, this theorem gives an algorithm to determine $H(P)$. We do not have such an algorithm for $\alpha$ having a conjugate on the unit circle, {\em e.g.}, when $\alpha$ is a Salem number.

The proofs of our theorems are detailed in 
Section~3. In Section~2 we show 
auxiliary results, some of which are used to prove 
Theorem~\ref{th:2}; these results are extensions 
of the corresponding ones of~\cite{K:99}.

\section{Some propositions}

We will follow the same steps as in~\cite{K:99} to show some auxiliary results of
independent interest. Note that our discussion is not restricted to expanding algebraic
integers; it is valid for general algebraic numbers.

For a non-zero algebraic number $\alpha,$ the set \ $\{0,\ldots,\left\vert
M_{\alpha}(0)\right\vert -1)\}$ is a CRS $\operatorname{mod}\alpha$ of the
ring $\mathbb{Z[\alpha]},$ and so any CRS contains exactly $\left\vert
M_{\alpha}(0)\right\vert $ elements. In other words, we identify
$\mathbb{Z}[\alpha]$ with $\mathbb{Z}[x]/(M_{\alpha})$ and consider its
quotient ring by an ideal $(x)$, that is isomorphic to $\mathbb{Z}%
[x]/(M_{\alpha},x)\simeq\mathbb{Z}/M_{\alpha}(0)\mathbb{Z}$. A CRS may be
identified with a set of representatives of the last quotient ring. Now fix a
CRS, say $R,$ then each element $\beta\in\mathbb{Z[\alpha]},$ can be written
in a unique way $\beta=r+\alpha\beta^{\prime},$ where $r\in R$ and
$\beta^{\prime}\in\mathbb{Z[\alpha]}$. Iterating the map 
\begin{equation}\label{eq:J}
\begin{array}{rrcl}
J:&\mathbb{Z}[\alpha]&\rightarrow& \mathbb{Z}[\alpha],\\[1mm]
&\beta&\mapsto &\displaystyle\frac{\beta-r}{\alpha},
\end{array}
\end{equation}
where $r$ is the unique element of $R$ satisfying $\beta\equiv r\operatorname{mod}%
\alpha,$ we can associate to any $\beta\in\mathbb{Z[\alpha]},$ a sequence
$(J^{(n)}(\beta))_{n\geq0}$ \ of elements of $\mathbb{Z[\alpha]},$ where
$J^{(0)}(\beta):=\beta$ $.$ In particular, if $J^{(n)}(\beta)=\beta$ for some
$n\geq1,$ then $\beta$ is said to be {\em periodic}; the set of periodic numbers is
denoted by $\wp.$ Setting $r_{n}=r_{n}(\beta):=J^{(n)}(\beta)-\alpha
J^{(n+1)}(\beta),$ where $n\geq0$ and $r_{n}\in R,$ we have
\begin{equation}\label{eq:(4)}
\beta=r_{0}+\cdot\cdot\cdot+r_{n}\alpha^{n}+\alpha^{n+1}J^{(n+1)}(\beta),
\end{equation}
and
\begin{equation}\label{eq:(5)}
J^{(n+1)}(\beta)=\frac{\beta}{\alpha^{n+1}}-\frac{r_{0}}{\alpha^{n+1}}%
-\cdot\cdot\cdot-\frac{r_{n}}{\alpha^{1}}.
\end{equation}

The following result gives some necessary and sufficient conditions for $S_{\alpha}$ to contain exactly one representative of each element of $\mathbb{Z}[\alpha]/\alpha\mathbb{Z}[\alpha]$.

\begin{proposition}\label{pr:1}
Let $\alpha$ be a non-zero algebraic number and let $R$ be a CRS of $\mathbb{Z}[\alpha]/\alpha\mathbb{Z}[\alpha]$. Then the following assertions are equivalent.
\begin{itemize}
\item[(i)] $\mathbb{Z}[\alpha]=R[\alpha].$
\item[(ii)] $\wp=\{J^{(n)}(0)\mid n\geq0\}$, and $\forall\beta\in\mathbb{Z}[\alpha],\; \exists\; s=s(\beta)\in\mathbb{N}$ such that $J^{(s+1)}(\beta)=0.$
\item[(iii)] $\wp=\{J^{(n)}(0)\mid n\geq0\},$ and $\forall\beta\in\mathbb{Z[\alpha]}$ the sequence $(J^{(n)}(\beta))_{n\geq0}$ is eventually periodic.
\end{itemize}
\end{proposition}

\begin{proof}
(i)$\Rightarrow$(ii). Suppose \textit{\ }$\mathbb{Z[\alpha]=}R\mathbb{[\alpha]}$ and
let $e_{0}+\cdot\cdot\cdot+e_{s}\alpha^{s}$ be a representation in $R[\alpha]$
of an element $\beta\in\mathbb{Z[\alpha]}$, where $s=s(\beta)\in\mathbb{N}.$
If $s=0,$ then by \eqref{eq:(4)} we have $\beta=r_{0}+\alpha J^{(1)}(\beta)=$
$e_{0}+\alpha0$ and so $J^{(1)}(\beta)=$ $0.$ Similarly, when $s\geq1$ we have
$\beta=r_{0}+\alpha(r_{1}+\cdot\cdot\cdot+r_{s}\alpha^{s-1}+\alpha
^{s}J^{(s+1)}(\beta))=e_{0}+\alpha(e_{1}+\cdot\cdot\cdot+e_{s}\alpha^{s-1}),$
$r_{0}=e_{0},$ $r_{1}+\cdot\cdot\cdot+r_{s}\alpha^{s-1}+\alpha^{s}%
J^{(s+1)}(\beta)=e_{1}+\cdot\cdot\cdot+e_{s}\alpha^{s-1},$ and by induction we
obtain $J^{(s+1)}(\beta)=0.$ It follows in particular when $\beta=0$ that
there is a positive integer $p=s(0)+1$ such that $J^{(p)}(0)=0;$ thus
$0\in\wp.$ Moreover, if $p$ designates the smallest integer satisfying the
last equality, then
\[
\{J^{(n)}(0)\mid n\geq0\}=\{J^{(n)}(0)\mid0\leq n\leq p-1\}\subset\wp,
\]
and for any $\beta\in\mathbb{Z[\alpha]},$ we have $\{J^{(n)}(\beta)\mid n\geq
s(\beta)+1\}=\{J^{(n)}(0)\mid0\leq n\leq p-1\}$; so $\wp\subset\{J^{(n)}%
(0)\mid n\geq0\}.$

(ii)$\Rightarrow$(iii) is trivial, since the relation
$0=J^{(s(0)+1)}(0)\in\wp$ gives that the sequence $(J^{(n)}(0))_{n\geq0}%
$\textit{\ }is purely periodic, and so we have, by the hypothesis
$J^{(s(\beta)+1)}(\beta)=0,$ where $\beta\in\mathbb{Z[\alpha]},$ that
$(J^{(n)}(\beta))_{n\geq0}$\ is eventually periodic. 

(iii)$\Rightarrow$(i). For each $\beta\in\mathbb{Z[\alpha]}$ there are two
positive rational integers $k$ and $m$ such that $J^{(k)}(\beta)=J^{(k+m)}%
(\beta).$ Hence, $J^{(m)}(J^{(k)}(\beta))=J^{(k)}(\beta),$ $J^{(k)}(\beta)\in$
$\wp$ and so $J^{(k)}(\beta)=J^{(l)}(0)$ for some $l\in\{0,\ldots,p-1\},$ where
$p$ is a positive rational integer such that $0=J^{(p)}(0);$ thus
$J^{(k+p-l)}(\beta)=J^{(p)}(0)=0,$ and by (5) we see that $\beta\in
R\mathbb{[\alpha]}.$
\end{proof}

\begin{corollary}\label{co:1} 
With the same assumption as in Proposition~\ref{pr:1} we have the following equivalence: $(\alpha,R)$ is a number system $\Longleftrightarrow\forall\beta\in\mathbb{Z[\alpha]},$ the sequence $(J^{(n)}(\beta))_{n\geq0}$ is eventually periodic, and
$\wp=\{0\}.$
\end{corollary}

\begin{proof}
The result is an immediate consequence of Proposition~\ref{pr:1}.
Indeed, if $(\alpha,R)$ is a number system, then $0\in R,$ $0=0+\alpha0,$
$J^{(1)}(0)=0$ and so $\wp=\{0\}.$ Conversely, if $\wp=\{0\},$ then $0\in\wp,$
$J^{(1)}(0)=0$ and by the relation (5) (with $n=0),$ we have that $0\in
R.$
\end{proof}

\begin{proposition}\label{pr:2}With the same hypothesis as in Proposition~\ref{pr:1}, for each $\beta\in\mathbb{Z}[\alpha]$ there is a constant $c=c(\alpha,\beta)\in\mathbb{N}$ and a positive integer $L=L(\alpha,R)$ such that $LJ^{(n)}(\beta)$ is an
algebraic integer, $\forall\, n\geq c.$
\end{proposition}

\begin{proof}
Clearly for any element $\gamma\in\mathbb{Z}[\alpha]$, there
is a positive integer $c=c(\alpha,\gamma)$ such that $\gamma/\alpha^{c}%
\in\mathbb{Z}[1/\alpha]$. Put $\ell=\max\{c(\alpha,r)\ |r\in R\}$. Then by (5)
and (6), for every $n\geq c(\alpha,\beta)$ we have $\alpha^{-\ell}%
J^{(n)}(\beta)\in\mathbb{Z}[1/\alpha]\cap\alpha^{-\ell}\mathbb{Z}[\alpha]$,
i.e., $J^{(n)}(\beta)\in\alpha^{\ell}\mathbb{Z}[1/\alpha]\cap\mathbb{Z}%
[\alpha]$. Letting $L$ be the absolute norm of the denominator of the
fractional ideal $(\alpha^{\ell})$ in $\mathbb{Q}(\alpha)$, we obtain
$LJ^{(n)}(\beta)\in L\alpha^{\ell}\mathbb{Z}[1/\alpha]\cap\mathbb{Z}[\alpha]$,
and we can deduce the result similarly to the proof of \cite[Lemma 3]{AZ:13}.
\end{proof}

For an algebraic number $\alpha$ we designate by $E(\alpha)$ the set of the
distinct embeddings of the field $\mathbb{Q}(\alpha)$\textit{\ }%
into\textit{\ }$\mathbb{C}.$ The following assertion may be easily deduced
from \cite[Lemma 1]{K:99}.

\begin{proposition}\label{pr:3}
Let $\alpha$ be an expanding algebraic number and let $R$ be a CRS. Then, there is a constant $c=c(\alpha,R)$ with the following property: for each $\beta\in\mathbb{Z}[\alpha]$ there is $n_0\in\mathbb{N}$  such that  $\left\vert\sigma(J^{(n)}(\beta))\right\vert \leq c$ for all $n\ge n_0$ and $\sigma\in E(\alpha).$
\end{proposition}

\begin{proof}
For each element $\sigma\in E(\alpha)$, set $K_{\sigma}%
:=\max\{\left\vert \sigma(r)\right\vert \mid r\in R\}.$ Then, by \eqref{eq:(5)}, we have
$\sigma(J^{(n+1)}(\beta))=\frac{\sigma(\beta)}{(\sigma(\alpha))^{n+1}}
-\frac{\sigma(r_{0})}{(\sigma(\alpha))^{n+1}}-\cdot\cdot\cdot-\frac
{\sigma(r_{n})}{(\sigma(\alpha))^{1}},$ \ $\left\vert \sigma(J^{(n+1)}
(\beta))\right\vert $ $\leq\frac{\left\vert \sigma(\beta)\right\vert
}{\left\vert \sigma(\alpha)\right\vert ^{n+1}}+\frac{K_{\sigma}}{\left\vert
\sigma(\alpha)\right\vert -1},$ and the result follows immediately by setting
(for example) $c(\alpha,R)$ the greatest value of the quantities
$1+\frac{K_{\sigma}}{\left\vert \sigma(\alpha)\right\vert -1},$ when $\sigma$
runs through $E(\alpha).$
\end{proof}

The first, second and last assertions of the corollary below, have been mentioned in \cite{K:99}, when $\alpha$ is an expanding integer.

\begin{corollary}\label{co:2}
Under the assumptions of Proposition~\ref{pr:3} the following assertions hold.
\begin{itemize}
\item[(i)] $\wp$ is a finite set.

\item[(ii)] $\forall\beta\in\mathbb{Z}[\alpha],$ the sequence $(J^{(n)}(\beta))_{n\geq0}$ is eventually periodic.

\item[(iii)] $\mathbb{Z[\alpha]=}R\mathbb{[\alpha]\ }\mathbb{\Leftrightarrow}\; \wp=\{J^{(n)}(0)\mid n\geq0\}.$

\item[(iv)] $(\alpha,R)$ is a number system $\Leftrightarrow \wp=\{0\}.$
\end{itemize}
\end{corollary}

\begin{proof}
We see that (i) and (ii) are consequences of the Propositions~\ref{pr:2} and \ref{pr:3}, and from this we deduce, by Proposition~\ref{pr:1}, the last two assertions.
\end{proof}

\section{Proofs of theorems}

\begin{proof}[Proof of Theorem~\ref{th:1}] Let $(\alpha,S)$ be a pair satisfying (1), and
fix for each element of $S$ a representation, say $%
%TCIMACRO{\dsum \limits_{k=0}^{s_{j}}}%
%BeginExpansion
{\displaystyle\sum\limits_{k=0}^{s_{j}}}
%EndExpansion
a_{k,j}\alpha^{k},$ where $j\in\{1,\ldots,\operatorname{Card}(S)\},$ $s_{j}\in\mathbb{N}$ and
$a_{k,j}\in\mathbb{Z}.$ Padding with zeros the last sums
may also be written \
\begin{equation}\label{eq:(6)}%
%TCIMACRO{\dsum \limits_{k=0}^{s}}%
%BeginExpansion
{\displaystyle\sum\limits_{k=0}^{s}}
%EndExpansion
a_{k,j}\alpha^{k}, %
\end{equation}
where $s:=\max\{s_{j}\mid1\leq j\leq \operatorname{Card}(S)\}.$ If $\beta=%
%TCIMACRO{\dsum \limits_{j=0}^{n}}%
%BeginExpansion
{\displaystyle\sum\limits_{j=0}^{n}}
%EndExpansion
\varepsilon_{j}\alpha^{j},$ where $n\in\mathbb{N}$ and $(\varepsilon
_{0},\ldots,\varepsilon_{n})\in S^{n+1},$ then we see, by \eqref{eq:(6)}, that
$\beta=R(\alpha)$ for some $R(x):=%
%TCIMACRO{\dsum \limits_{k=0}^{D}}%
%BeginExpansion
{\displaystyle\sum\limits_{k=0}^{D}}
%EndExpansion
A_{k}x^{k}\in\mathbb{Z}[x]$ and $D\geq s.$ Moreover, a short computation shows
that the values of $A_{k},$ are among the numbers $a_{0,j_{0}}+a_{1,j_{1}%
}+\cdot\cdot\cdot+a_{\min\{k,s\},j_{\min\{k,s\}}},$ where $(j_{0}%
,\ldots,j_{\min\{k,s\}})$ $\in\{1,\ldots,\operatorname{Card}(S)\}^{\min\{k,s\}+1}.$ Hence, the
number of possible values of the coefficients of $R$ is given by $ 
%TCIMACRO{\dsum \limits_{k=1}^{s+1}}%
%BeginExpansion
{\displaystyle\sum\limits_{k=1}^{s+1}}
%EndExpansion
\operatorname{Card}(S)^{k}<\infty,$ and so $\alpha$ satisfies HRP. It follows immediately from
\cite[Theorem~1]{AZ:13}, that $\alpha$ is an algebraic number. Now, we show that $\operatorname{Card}(S)\geq2.$ Clearly $S[\alpha]=\{0\}\neq\mathbb{Z}[\alpha]$ when $S=\{0\}.$ The representation $%
%TCIMACRO{\dsum \limits_{j=0}^{n}}%
%BeginExpansion
{\displaystyle\sum\limits_{j=0}^{n}}
%EndExpansion
s\alpha^{j}$ of $0,$ must exist when $\left\{  s\right\}  =S\neq\{0\}.$
However this implies that $\alpha$ is a root of unity not equal to $1,$ and so
$S[\alpha]$ is a bounded subset of $\mathbb{C}.$ This means that
$S[\alpha]\neq\mathbb{Z}[\alpha].$ Hence, $\operatorname{Card}(S)\geq2,$ and the first
inequality in Theorem~\ref{th:1} is true, as
\[%
%TCIMACRO{\dsum \limits_{k=1}^{s+1}}%
%BeginExpansion
{\displaystyle\sum\limits_{k=1}^{s+1}}
%EndExpansion
\operatorname{Card}(S)^{k}=\operatorname{Card}(S)(\operatorname{Card}(S)^{s+1}-1)/(\operatorname{Card}(S)-1).
\]
To end the proof of Theorem~\ref{th:1} assume without loss of generality that
$\left\vert M_{\alpha}(0)\right\vert \geq2.$ 
For each $\beta\in\mathbb{Z}[\alpha]$ the representation $\beta = a_0+a_1\alpha+\cdots+a_L\alpha^L \in S_\alpha[\alpha]$ has $a_0 \equiv \beta \operatorname{mod}\alpha$. Thus, $S_\alpha$ contains a complete system of coset representatives of $\mathbb{Z}[\alpha]/\alpha\mathbb{Z}[\alpha]$. Thus by Gauss' Lemma for each $j\in\mathbb{Z}$ there is $P(\alpha) \in S_\alpha$ such that $P(0)\equiv j \operatorname{mod}M_\alpha(0)$.
%Then any rational integer $N$ may
%be written $N=%
%%TCIMACRO{\dsum \limits_{j=0}^{n}}%
%%BeginExpansion
%{\displaystyle\sum\limits_{j=0}^{n}}
%%EndExpansion
%s_{j}\alpha^{j}$ for some $n\in\mathbb{N}$ and $(s_{0},\ldots,s_{n})\in S^{n+1}.$
%Let $P(\alpha)$ is a fixed representation of $s_{0}$ in $\mathbb{Z}[\alpha],$
%where $P\in$\textit{\ }$\mathbb{Z}[x]$ be such that replacing each $s_{j}$
%with $j\in\{1,\ldots,n\},$ by one of its representations in $\mathbb{Z}[\alpha],$
%the last equality gives that $N-P(0)=\alpha\beta,$ for some $\beta\in
%$\textit{\ }$\mathbb{Z}[\alpha].$ It follows by Gauss' lemma that $P(0)\equiv
%N\operatorname{mod}M_{\alpha}(0),$ and for each $N\in\{0,\ldots,\left\vert
%M_{\alpha}(0)\right\vert -1\},$ there is $P_{N}\in$\textit{\ }$\mathbb{Z}[x]$
%with $P_{N}(\alpha)\in S$ and $P_{N}(0)\equiv N\operatorname{mod}M_{\alpha
%}(0).$ From this a simple computation gives that $\operatorname{Card}(S)\geq\left\vert
%M_{\alpha}(0)\right\vert .$
\end{proof}

\begin{proof}[Proof of Theorem~\ref{th:2}~(i)] It is clear that $\mathbb{Z}[1]=\mathbb{Z=}%
\left\{  -1,1\right\}[1].$ Let $\alpha\neq1$ be a root of unity, and let
$m\in\mathbb{N\cap}[2,\infty)$ satisfying $\alpha^{m}=1$. Then, 
%$0=%
%TCIMACRO{\dsum \limits_{j=0}^{m-1}}%
%BeginExpansion
%{\displaystyle\sum\limits_{j=0}^{m-1}}
%EndExpansion
%\alpha^{j},$ and
using the fact that $\alpha^{jm}=1,$ where $j\in\mathbb{N},$
a simple induction shows that $\mathbb{N\subset}\left\{  0,1\right\}
[\alpha].$ Similarly, by the equation $-1=%
%TCIMACRO{\dsum \limits_{j=1}^{m-1}}%
%BeginExpansion
{\displaystyle\sum\limits_{j=1}^{m-1}}
%EndExpansion
\alpha^{j},$we obtain that every negative rational integer belongs to the set
$\left\{  0,1\right\}  [\alpha].$ After this, a simple induction on the degree
of the representations of the elements of $\mathbb{Z}[\alpha]$, leads
immediately to the equation $\mathbb{Z}[\alpha]=\left\{  0,1\right\}
[\alpha].$

Now, consider $\alpha\in%
%TCIMACRO{\tciFourier}%
%BeginExpansion
\mathcal{F}%
%EndExpansion
_{2}$ which is not a root of unity. Then, the relation \eqref{eq:(2)} gives that
$\left\vert M_{\alpha}(0)\right\vert \leq2,$ and so by \cite[Theorem~1]{AZ:13}, we
obtain that $\alpha$ is an expanding integer, or is an algebraic number whose
conjugates are of modulus 1, with $\left\vert M_{\alpha}(0)\right\vert =2$.
Indeed, if $\alpha$ is an expanding number which is not an expanding integer,
then the leading coefficient, say $c,$ of $M_{\alpha},$ satisfies $\left\vert
c\right\vert \geq2,$ and so
\[
\left\vert \frac{M_{\alpha}(0)}{c}\right\vert \leq\frac{2}{2};
\]
this last inequality leads to a contradiction, because the absolute value of
the product of the conjugates of $\alpha$ is greater than 1. 
\end{proof}

\begin{proof}[Proof of Theorem~\ref{th:2} (ii)] Suppose that $\alpha\in%
%TCIMACRO{\tciFourier}%
%BeginExpansion
\mathcal{F}%
%EndExpansion
_{\left\vert M_{\alpha}(0)\right\vert }.$ Then, $\alpha$ satisfies HRP. By
Theorem 1, we see that $\left\vert M_{\alpha}(0)\right\vert \geq2,$ and any
corresponding set $S_{\alpha},$ satisfies $\operatorname{Card}(S_{\alpha})=\left\vert
M_{\alpha}(0)\right\vert ;$ thus $S_{\alpha}$ is a complete residue system
$\operatorname{mod}\alpha$ in $\mathbb{Z[\alpha]}.$ Set $M_{\alpha}%
(x)=A_{0}+A_{1}x+\cdot\cdot\cdot+A_{d}x^{d},$ and assume on the contrary that
$M_{\alpha}(1)^{2}=1.$ We shall obtain a contradiction by considering the
non-zero number
\[
\beta_{0}:=\frac{M_{\alpha}(1)}{(1-\alpha)}.
\]
Indeed, a simple computation shows that
\[
\beta_{0}=%
%TCIMACRO{\dsum \limits_{j=0}^{d-1}}%
%BeginExpansion
{\displaystyle\sum\limits_{j=0}^{d-1}}
%EndExpansion
\alpha^{j}%
%TCIMACRO{\dsum \limits_{k=j+1}^{d}}%
%BeginExpansion
{\displaystyle\sum\limits_{k=j+1}^{d}}
%EndExpansion
A_{k},
\]
and so $\beta_{0}\in\mathbb{Z}[\alpha].$ Moreover, if we fix a non-zero
element $s\in S_{\alpha},$ and we set
\[
\beta:=sM_{\alpha}(1)\beta_{0},
\]
then $\beta\in\mathbb{Z}[\alpha],$ $\beta=s+\alpha\beta,$ and so
$J^{(1)}(\beta)=\beta;$ thus $J^{(n)}(\beta)=\beta$ for all $n\geq0,$ and by
Proposition 4 we obtain a contradiction, since $\beta\neq0.$
\end{proof}

\begin{remark}\label{re:1}
It follows immediately by Theorem 2 (ii), that $2\in\mathcal{F}_{3},$ since $M_{2}(1)=1-2\Rightarrow2\notin\mathcal{F}_{2},$ and $\{-1,0,1\}[2]=\mathbb{Z}\Rightarrow \operatorname{Card}(S_{2})\leq3$ (this relation may also be deduced from Theorem~\ref{th:2}~(v)), and so, by Theorem~\ref{th:2}~(i), we have  $\mathbb{Q}\cap\mathcal{F}_{2}=\{-2,-1,1\}.$ Concerning the quadratic case the results of Gilbert~\cite{G:81} and  K\'{a}tai and Kov\'{a}cs~\cite{KK:81} imply that there are at least $8$ quadratic (non-real) expanding integers in $\mathcal{F}_{2}$.
Also, a short computation shows that if $\alpha$ is an expanding real quadratic integer satisfying $|M_{\alpha}(0)|= 2,$ then $M_{\alpha}(x)=x^{2}-2,$ and so $\alpha\notin\mathcal{F}_{2},$ as $|M_{\alpha}(1)|=1.$ 
For higher degrees consider for example the $p-$Eisenstein polynomials $x^{d}+px^{k}+\cdot\cdot\cdot+p,$ where $p$ is a prime, $d\geq2$, and $k$ runs through $\{0,1,\ldots,d-1\}$. We see, by the above mentioned result of Kov\'{a}cs~\cite{K:81}, that each set $\mathcal{F}_{p}$ contains at least $d^{2}$ expanding integers with degree $d.$ On the other hand, if $d\equiv0 \operatorname{mod}2$ and $\alpha$ is a root of $M_{\alpha}(x)=x^{d}+px^{d-1}+\cdot\cdot\cdot+px+p,$ then $-\alpha$ is of degree $d$ and satisfies $-\alpha$ $\notin\mathcal{F}_{p}$ as $M_{-\alpha}(1)=1$. 
\end{remark}

\begin{remark}\label{re:2} 
It is easy to see when $\alpha\in\mathcal{F}_{2}$ and $\alpha$ is not an algebraic integer, then $0\notin S_{\alpha}$ (that is $S_{\alpha}$ is not a number system). Indeed, if $0\in S_{\alpha}$
and $b$ designates the other (non-zero) element of $S_{\alpha},$ then the
representation of $-b$ in $S_{\alpha}[\alpha]$ is of the form $b+b\alpha
^{n_{1}}+\cdot\cdot\cdot+b\alpha^{n_{s}},$ where $1\leq n_{1}<$ $\ldots<n_{s}$. Thus $\alpha$ is a root of the polynomial $2+x^{n_{1}}+\cdot\cdot\cdot+x^{n_{s}}$, contradicting the fact that $\alpha$ is not an algebraic integer.  Notice that Theorem~\ref{th:2}~(ii) can also be applied in the non-integral case. For example, if $M_{\alpha}(x)=2x^{2}-3x+2,$ then $\alpha$ is a quadratic algebraic number whose conjugates are of modulus 1 ($\alpha$
satisfies HRP by \cite[Theorem 2]{AZ:13}) and $\alpha\notin\mathcal{F}_{2},$ since $M_{\alpha}(1)=1.$
\end{remark}

\begin{proof}[Proof of Theorem~\ref{th:2}~(iii)] Let\textit{ }$\alpha$ be an algebraic
number whose conjugates are all of modulus 1, such that $\operatorname{Card}(S_{\alpha
})=\left\vert M_{\alpha}(0)\right\vert.$ Then, $\alpha$ is not an algebraic
integer, since otherwise $\alpha$ is a root of unity (by Kronecker's theorem),
and Theorem 2 (i) gives, in this case, that $\operatorname{Card}(S_{\alpha})=2>1=\left\vert
M_{\alpha}(0)\right\vert .$ Notice also that $M_{\alpha}(x)=x^{d}M_{\alpha
}(1/x),$ where $d$ is the degree of $M_{\alpha},$ and so the leading
coefficient of $M_{\alpha}$, say $c$, satisfies $M_{\alpha}(0)=c\geq2.$ Assume
on the contrary, that there is a set $S_{\alpha}\subset\mathbb{Z}$ satisfying
$\operatorname{Card}(S_{\alpha})=c.$ Then, by Theorem~\ref{th:1}, the set $S_{\alpha}$ is a CRS
$\operatorname{mod}\alpha$ in $\mathbb{Z[\alpha]},$ and is also a CRS
$\operatorname{mod}c$ in $\mathbb{Z}.$ Now, we claim that the map $J$, which was defined
on the ring $\mathbb{Z[\alpha]}$ in \eqref{eq:J}, restricted to the set
\[
U:=\mathbb{Z}[\alpha]\cap\frac{1}{\alpha}\mathbb{Z}\left[  \frac{1}{\alpha
}\right]  ,
\]
is a bijection of $U$. Indeed, if $\beta\in U$ and $r=r(\beta)$ is the unique
element of $S_{\alpha}$ such that
\[
J(\beta)=(\beta-r)/\alpha,
\]
then $(\beta-r)\in\mathbb{Z[}1/\alpha\mathbb{]},$ $(\beta-r)/\alpha
\in(1/\alpha)\mathbb{Z[}1/\alpha\mathbb{]}$,\ and $J(\beta)\in U;$ thus
$J(U)\subset U.$ Moreover, if $\sum_{j=1}^{s}a_{j}/\alpha^{j}$ and $\sum
_{j=1}^{t}b_{j}/\alpha^{j}$ designate, respectively, some representations in
$(1/\alpha)\mathbb{Z[}1/\alpha\mathbb{]}$ of two elements $\beta$ and $\gamma$
of $U$, then the equation $J(\beta)=J(\gamma)$ gives immediately that $\alpha$
is a root of a polynomial with integer coefficients, whose leading
coefficient is $(r(\gamma)-r(\beta)).$ It follows by Gauss' Lemma that $r(\gamma)\equiv r(\beta)\operatorname{mod}c$ and
so $r(\gamma)=r(\beta);$ thus $\beta=\gamma$ and $J$ is injective. To complete
the proof of the claim, fix again a representation $\sum_{j=1}^{v}c_{j}%
/\alpha^{j},$ in $(1/\alpha)\mathbb{Z[}1/\alpha\mathbb{]},$ of an element $y$
of $U.$ Then, the relation $\alpha y-c_{1}\in(1/\alpha)\mathbb{Z[}%
1/\alpha\mathbb{]},$ together with the equality $-c_{1}=r(-c_{1})-ck,$ where
$k\in\mathbb{Z},$ yield $\alpha y+r(-c_{1})\in ck+(1/\alpha)\mathbb{Z[}%
1/\alpha\mathbb{]},$ and so
\begin{equation}\label{eq:(7)}
\alpha y+r(-c_{1})\in(1/\alpha)\mathbb{Z[}1/\alpha\mathbb{]}, 
\end{equation}
as $M_{\alpha}(\alpha)=0\Rightarrow c\in(1/\alpha)\mathbb{Z[}1/\alpha
\mathbb{]}$ and
\begin{equation}\label{eq:(8)}
c\mathbb{Z}\subset(1/\alpha)\mathbb{Z[}1/\alpha\mathbb{]}. 
\end{equation}
Since $\alpha y+r(-c_{1})\in\mathbb{Z[}\alpha\mathbb{]},$ $r(-c_{1})\in
S_{\alpha}$ and $J(\alpha y+r(-c_{1}))=y,$ we see by \eqref{eq:(7)} that $J$ is a
surjective, and so $J$ is a bijection of $U$. Notice also that $U$ and
$S_{\alpha}$ have only one common element (which is the unique element in $c\mathbb{Z}\cap S_\alpha$). It follows immediately that $J^{(n)}(0)\in U,$ for all $n\in\mathbb{N}$, and so $\wp=\{J^{(n)}(0)\mid n\geq0\}\subset U$. \ Recall,
by Proposition~\ref{pr:1}, that $\wp$ is finite and for each\textit{\ }$\beta\in
U,$ there exists $s\geq1$ such that $J^{(s)}(\beta)=0$. Moreover, as $\wp \subset U$ is finite and $J$ is bijective on $U$ there is $t\in \mathbb{N}$ such that  $J^{(t)}(0)=0$. Thus, again by bijectivity of $J$ the number $\beta$ has to occur somewhere in the cycle (each arrow indicates an application of $J$)
\[
0 \xrightarrow{J} y_1 \xrightarrow{J} y_2  \xrightarrow{J} \cdots  \xrightarrow{J} y_{t-1}   \xrightarrow{J} 0
\]
and, hence, $J^{(t-s)}(0)=\beta$ which implies that $\beta \in U$. Thus $U\subset\wp$, a contradiction, because by \eqref{eq:(8)} we have that $c\mathbb{Z}\subset$ $U$ and so $U$ cannot be finite.
\end{proof}

\begin{proof}[Proof of Theorem~\ref{th:2}~(iv)] 
Since the eigenvalues of the companion
matrix of the polynomial $M_{\alpha}$ are all of modulus greater than 2, the
result follows immediately from \cite[Theorem 4]{GK:07}.
\end{proof}

$\bigskip$

\begin{proof}[Proof of Theorem~\ref{th:2}~(v)]  It is clear that $M_{\alpha}(x)=bx-a,$ and if
$R$ is a CRS $\operatorname{mod}a$ in $\mathbb{Z},$ then $R$ is also a CRS
$\operatorname{mod}\alpha$ in $\mathbb{Z}[\alpha],$ as $a=b\alpha.$ Suppose
first that\textit{ }$a\neq b+1.$ We shall show that there is a CRS
$\operatorname{mod}a$ in $\mathbb{Z},$ say $S$, such that every element
$\beta\in$ $\mathbb{Z}[\alpha]$ may be written $\beta=\varepsilon_{0}%
+\cdot\cdot\cdot+\varepsilon_{n}\alpha^{n},$ for some $n\in\mathbb{N}$ and
$(\varepsilon_{0},\ldots,\varepsilon_{n})\in S^{n+1}.$ For this purpose it is
enough to prove the inclusion
\begin{equation}\label{eq:(9)}
\mathbb{Z\subset}\text{ }S[\alpha]. 
\end{equation}
Indeed, assume that all elements of $\mathbb{Z}[x]$ with degree at most $d-1,$
where $d\geq1,$ evaluated at $\alpha$ belong to the set $S[\alpha],$ and let
$P(x)=a_{0}+a_{1}x+\cdot\cdot\cdot+a_{d}x^{d}\in\mathbb{Z}[x].$ Since the
constant term $a_{0}$ may be written
\[
a_{0}=\varepsilon+aa_{1}^{\prime}=\varepsilon+ba_{1}^{\prime}\alpha,
\]
for some $\varepsilon\in S,$ and $a_{1}^{\prime}\in\mathbb{Z},$ we see that
\[
P(\alpha)=\varepsilon+(a_{1}+ba_{1}^{\prime})\alpha+\cdot\cdot\cdot
+a_{d}\alpha^{d}=\varepsilon+\alpha Q(\alpha),
\]
where $Q(x)=(a_{1}+ba_{1}^{\prime})+\cdot\cdot\cdot+a_{d}x^{d-1}\in
\mathbb{Z}[x],$ and by the induction hypothesis we obtain the result.
Furthermore, to show the inclusion \eqref{eq:(9)}, it suffices to prove that $ka\in\alpha
S[\alpha],$ $\forall$ $k\in\mathbb{Z},$ or equivalently
\begin{equation}\label{eq:(10)}
kb\in S[\alpha],\text{ \ }\forall k\in\mathbb{Z},
\end{equation}
as any rational integer may be written $\varepsilon+ka$ for some
$\varepsilon\in S$ and $k\in\mathbb{Z}.$

Assume first that $0< -b < a$ and choose $S_\alpha=\{0,\ldots, a-1\}$. As $S_\alpha$ is a CRS modulo $a$ the mapping $J$ in \eqref{eq:J} is well-defined and as in \eqref{eq:(4)} we can be used to attach a formal sum
\[
kb= \sum_{i=0}^\infty d_i\alpha^i
\]
to each $kb$ with $k\in \mathbb{Z}$. We denote this by $kb=(\ldots,d_1,d_0)_\alpha$. To prove that $b\mathbb{Z}\subset S_{\alpha}[\alpha]$ we need to show that for each $k\in\mathbb{Z}$ there is $\ell\in\mathbb{N}$ such that $a_i=0$ for each $i\ge \ell$ (in this case we say that $kb$ has a {\em finite expansion}). This will be done by an induction involving a so-called {\em transducer} automaton (see {\it e.g.}~\cite{E:74} for the definition of these objects).

As $0=(\ldots,0,0)_\alpha$ it is clear that $0$ has a finite expansion. Now assume that $kb$ has a finite expansion for some given $k$. To conclude the induction proof we have to show that the same is true for $(k \pm 1)b$. To this matter we study the effect of ``addition and subtraction of $b$'' on the expansion of a number.

We first deal with the addition of $b$. Let $kb=(\ldots ,d_2,d_1,d_0)_\alpha$. If $d_0+b \in S_\alpha$, then $(k+1)b=(\ldots d_2,d_1,d_0+b)_\alpha$ and we are done. If, however, $d_0+b \not\in S_\alpha$, then certainly $d_0+b+a \in S_\alpha$, and, 
observing that $b\alpha - a=0$, we get that $(k+1)b=(\ldots ,d_2,d_1-b,d_0+b+a)_\alpha$. In this case, again two things can happen: either $d_1-b \in S_\alpha$, in which case we are done, or $d_1-b-a \in S_\alpha$, in which case we gain $(k+1)b=(\ldots ,d_2+b,d_1-b-a,d_0+b+a)_\alpha$ and have to go on again.  Subtraction of $b$ is treated analogously. As in Akiyama {\em et al.}~\cite{AFS:08} we use a transducer to model this process (see Figure~\ref{fig:trans1}).

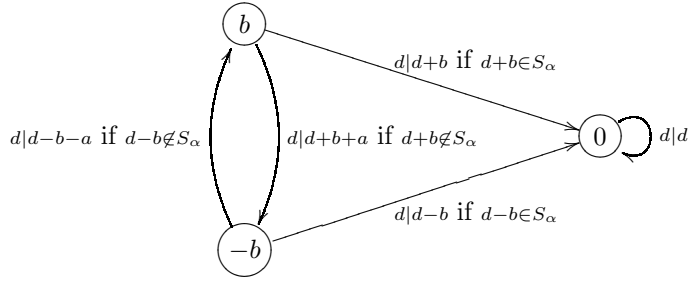
\begin{figure}[ht]
\hskip 0.5cm
\xymatrix{
*++[o][F]{b} \ar[rrrrd]^{\hskip 0.5cm d|d+b \hbox{ if }d+b\in S_\alpha} \ar@/^3ex/[dd]^{\hskip 0.0cm d|d+b+a \hbox{ if }d+b\not\in S_\alpha}&&&& \\
&&&&*++[o][F-]{0}\ar@(ur,rd)[]^{\hskip 0.01cm d|d}\\
*++[o][F-]{-b}\ar[rrrru]_{\hskip 0.5cm d|d-b \hbox{ if }d-b\in S_\alpha}\ar@/^3ex/[uu]^{\hskip 0.0cm d|d-b-a \hbox{ if }d-b\not\in S_\alpha}&&&&
}
\caption{The transducer for $0< -b < a$. If we use $b$ as a starting state, this transducer reads the digits of $kb$ from right to left and writes out the digits of $(k+1)b$. If we use $-b$ as starting state, it writes out the digits of $(k-1)b$. The label $d_1|d_2$ means that reading $d_1$ the transducer writes out $d_2$. As soon as we arrive at the state $0$, the remaining digits are just copied. \label{fig:trans1}}
\end{figure}

Let $kb=(\ldots,d_1,d_0)_\alpha$ be given. Feeding the digits of this expansion in the transducer depicted in Figure~\ref{fig:trans1} from right to left starting at the state $\pm b$ the transducer will write out the digits of the expansion of $(k\pm 1)b$. Since, by the induction assumption, $kb$ has a finite expansion, eventually we read only the digit $0$. However, as it is easily seen that two zeros in a row make sure that the transducer arrives in the ``accepting state'' 0, we conclude that the length of the expansion of $(k\pm 1)b$ can be at most by two (nonzero) digits longer than the expansion of $kb$. This proves that the expansion of $(k\pm 1)b$ is finite and the induction proof is finished.

Let now $0< b < a-1$. In this case the choice of the digit set is a bit more subtle; the multiples of $a-b$ play a special role here and need to be shifted to the negative by $a$. Indeed, set $R=\{0,\ldots,a-1\}$ and $B=\{k(a-b), k(a-b)-a\mid  1\le k \le \frac{a-1}{a-b}\}$. Then a convenient digit set is given by the symmetric difference $S_\alpha = R \triangle B$. 
The following assertions are easily checked.
\begin{itemize} 
\item[(A)] $S_\alpha$ is a CRS modulo $a$.
\item[(B)] If $d\in S_\alpha$ then either $d+b \in S_\alpha$ or $d+b-a\in S_\alpha$. 
\item[(C)] If $d\in S_\alpha$ then either $d-b \in S_\alpha$ or $d-b-a\in S_\alpha$. 
\item[(D)] $\{-b,b\} \subset S_\alpha$ (here we have to use that $b\not= a-1$).
\end{itemize}
The mapping $J$ is well defined by (A). Moreover, (B) and (C) make sure that the transducer in Figure~\ref{fig:trans2} can process all digit strings $(\ldots ,d_2,d_1,d_0) \in \{0,1,\ldots,a-1\}^\mathbb{N}$ and produces a well-defined unique output. 
\begin{figure}[ht]
\hskip 0.5cm
\xymatrix{
&*++[o][F]{b} \ar[rrrrd]^{\hskip 0.5cm d|d+b \hbox{ if }d+b\in S_\alpha} \ar@(l,d)[]_{\hskip 0.0cm d|d+b-a \hbox{ if }d+b\not\in S_\alpha}&&&& \\
&&&&&*++[o][F-]{0}\ar@(ur,rd)[]^{\hskip 0.01cm d|d}\\
&*++[o][F-]{-b}\ar[rrrru]_{\hskip 0.5cm d|d-b \hbox{ if }d-b\in S_\alpha}\ar@(l,u)[]^{\hskip 0.0cm d|d-b+a \hbox{ if }d-b\not\in S_\alpha}&&&&
}
\caption{The transducer for $0< b < a-1$. \label{fig:trans2}}
\end{figure}
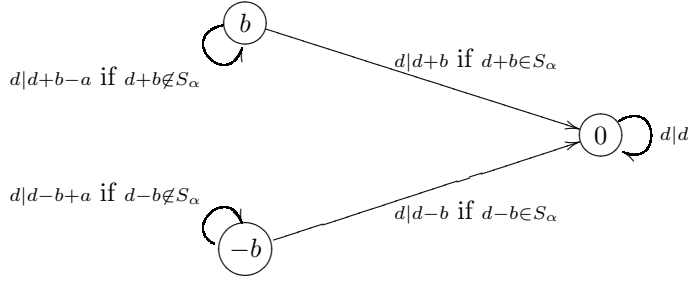
Indeed, direct calculations similar to the ones presented in the case $0<-b<a$ show that this transducer performs the addition of $\pm b$ to expansions $kb=(\ldots ,d_2,d_1,d_0)_\alpha$. Again we can now use induction to show that $kb$ has finite expansion for all $k\in \mathbb{Z}$.  Now (D) implies that two zeros in a row make sure that the transducer arrives in the ``accepting state'' 0 and we conclude again that the length of the expansion of $(k\pm 1)b$ can be at most by two (nonzero) digits longer than the expansion of $kb$. 

It remains to deal with the case $0<b = a-1$. Let $S_\alpha$ be given and suppose that $d \in S_\alpha$ is nonzero. Then $J(-bd) = \frac{-bd-d}{\alpha}=-bd$. Thus, to get a finite expansion of $-bd$ we need another element of $S_\alpha$ that lies in the same residue class modulo $a$. Therefore, for each nonzero residue class we need at least two representatives to be contained in $S_\alpha$ in order to guarantee that $b\mathbb{Z} \subset S_{\alpha}[\alpha]$ and, hence, $|S_\alpha|>2|M_\alpha(0)|-1$. As it follows immediately from the results on the case $0<-b<a$ that $S_\alpha=\{-a+1,\ldots,a-1\}$ satisfies $S_\alpha[\alpha]=\mathbb{Z}[\alpha]$, we conclude that $\alpha \in \mathcal{F}_{2|M_{\alpha}(0)|-1}$.
\end{proof}

\begin{proof}[Proof of Theorem 3] 
Let $\alpha=\alpha_{1},$ $\alpha_{2},\dots,$
$\alpha_{d}$ be the conjugates of $\alpha,$ arranged so that $|\alpha_{k}|>1$
for $k=1,\dots,n$ and $|\alpha_{k}|<1$ for $k=n+1,\dots,d.$ For each $k$ we
denote by $x_{k}$ the corresponding conjugate of any element $x$ in
$\mathbb{Z}[\alpha]$. Let $\mathfrak{p}_{1},\dots,\mathfrak{p}_{s}$ designate
the prime ideals which appear in the denominator of the prime ideal
decomposition of $(\alpha)$ in $\mathbb{Q}(\alpha)$. Set
\[
S(H):=\left\{  x\in\mathbb{Z}[\alpha]\ \left\vert \ |x_{j}|\leq\frac
{H}{|1-|\alpha_{j}||},\;(j=1,\dots ,d),\;\; \mu_{j}(x)\geq0\quad(j=1,\dots,s)\right.  \right\}  ,
\]
where $\mu_{j}(x)$ is the discrete valuation of $x$ lying over $\mathfrak{p}%
_{j}$. By definition $S(H)$ is a finite set, because it is a subset of
$\mathbb{Z}[\alpha]$ whose elements have bounded denominators and conjugates.
Clearly $S(H)\subset S(H+1)$. Let $S_{1}=\{0\}$ and we inductively define
\[
S_{j+1}=\left\{  \alpha y+d\ \left\vert \
\begin{array}
[c]{lll}%
d\in\{-H,\dots,H\},\ y\in S_{j}, &  & \\
|\alpha_{k}y_{k}+d|\leq \frac
{H}{|1-|\alpha_{k}||}\text{ for }k\leq n, &  & \\
\mu_{j}(\alpha y+d)\geq0\text{ for }1\leq j\leq s &  &
\end{array}
\right.  \right\}  
\]
Then we easily see that 
\[
S_{\infty}=\bigcup_{j=1}^{\infty}S_{j}%
\]
is a subset of $S(H)$. 
Construct an automaton $Z(H)$ having
states $S_{\infty}$, transitions $\delta(y,d)=\alpha y+d$ 
are defined if there exists a $j$ with
$y\in S_j$ and $\alpha y+d\in S_{j+1}$, and $0$
is both an initial and a final state. 
We claim
that this automaton has the required property. In fact, assume that
$\sum_{j=0}^{m}d_{j}\alpha^{j}=0$ with $d_{j}\in\{-H,\dots,H\}$. It is obvious
that
\[
\left\vert \sum_{j=u}^{m}d_{j}\alpha_{k}^{j-u}\right\vert \leq\frac
{H}{1-|\alpha_{k}|}%
\]
holds for $k\geq n+1$ and $u=0,\dots,m.$ For $k\leq n$, note that if
$|x_{k}|>H/(|\alpha_{k}|-1)$ and $x_{k}\in\mathbb{Q}(\alpha_{k}),$ then
$|\alpha_{k}x_{k}-d|>H/(|\alpha_{j}|-1)$ for $|d|\leq H$. In plain words, this
means that once $x_{k}$ becomes larger than $H/(|\alpha_{k}|-1)$, then there
is no way to come back to zero. Thus we see that
\[
\left\vert \sum_{j=u}^{m}d_{j}\alpha_{k}^{j-u}\right\vert \leq\frac
{H}{|1-|\alpha_{k}||}%
\]
is valid for all $u=0,\dots,m$ and $k=1,\dots,d$. Similarly since $\mu
_{k}(x)<0$ implies the relation $\mu_{k}(\alpha x-d)=\mu_{k}(\alpha x)<\mu
_{k}(x)<0$, we see that
\[
\mu_{k}\left(\sum_{j=u}^{m}d_{j}\alpha_{k}^{j-u}\right)\geq0
\]
for all $u$ and $k$. Therefore the sequence of states $\sum_{j=u}^{m}%
d_{j}\alpha_{k}^{j-u}$ with $u=m,m-1,\dots,0$ gives a successful path of
$Z(H)$ whose output is $d_{m}d_{m-1}\dots d_{0}$. 
\end{proof}

\begin{remark} 
By standard operations in automata, we can recognize the
set of the mirrored words $d_{0}d_{1}\dots d_{m}$. 
The automaton constructed
in the proof of Theorem~\ref{th:3} is not \textit{trim}, i.e., 
%some states may not be
%reachable from $0$ and 
$0$ may not be reachable from some states, but it is
easy to make it to a trim automaton.
\end{remark}

\bigskip

\textbf{Acknowledgment.} We thank the referee for careful reading of this paper.

\end{document}